\documentclass{dcds-sEA} 
\usepackage{amsmath}
\usepackage{paralist}
\usepackage[misc]{ifsym}
\usepackage{epsfig} 
\usepackage{epstopdf} 
\usepackage[colorlinks=true]{hyperref}
\hypersetup{urlcolor=blue, citecolor=red}
\allowdisplaybreaks

\def\currentvolume{X}
 \def\currentissue{X}
  \def\currentyear{200X}
   \def\currentmonth{XX}
    \def\ppages{X-XX}
     \def\DOI{10.3934/dcdss.2024035}

\newtheorem{theorem}{Theorem}[section]

\newtheorem{lemma}[theorem]{Lemma}

\theoremstyle{definition}

\newtheorem{remark}[theorem]{Remark}

\title[Normalized ground states for nonlinear Schr\"{o}dinger  equations]
{Normalized ground state solutions for nonlinear Schr\"{o}dinger equations with general Sobolev critical nonlinearities} 

\author[Manting Liu and Xiaojun Chang]{}


\thanks{The second author is supported by  NSFC (11971095).}

\thanks{$^*$Corresponding author: Xiaojun Chang}

\begin{document}
\maketitle

\centerline{\scshape
Manting Liu$^{{\href{mailto:liumt679@nenu.edu.cn}{\textrm{\Letter}}}}$
and Xiaojun Chang$^{{\href{mailto:changxj100@nenu.edu.cn}{\textrm{\Letter}}}*}$}

\medskip

{\footnotesize
 \centerline{School of Mathematics and Statistics \& Center for Mathematics and Interdisciplinary Sciences}
   \centerline{ Northeast Normal University, Changchun 130024,
China}
} 






\begin{abstract}
\noindent This paper is concerned with the existence of normalized solutions for nonlinear Schr\"{o}dinger equations.
The nonlinearity has a Sobolev critical growth at infinity but does not satisfy the Ambrosetti-Rabinowitz condition.
By analysing the monotonicity of the ground state energy with respect to the prescribed mass $c$, we employ the constrained minimization approach and concentration-compactness principle to establish the existence of normalized ground state solutions for all $c>0.$
%
\end{abstract}

\section{Introduction and main results}

\renewcommand\theequation{1.\arabic{equation}}
In this paper, we deal with the following nonlinear Schr\"{o}dinger equation
\begin{equation}\label{eq01}
-\Delta u=f(u)+ \lambda u\quad  \mbox{in}\  \mathbb{R}^{N}
\end{equation}
under the constraint
\begin{equation}\label{eq02}
\int_{\mathbb{R}^N}|u|^2dx=c,
\end{equation}
where $N\ge3$, $c>0$, $f\in C^1(\mathbb{R}, \mathbb{R})$ and $\lambda\in \mathbb{R}$ is a Lagrange multiplier.
A function $u\in H^1(\mathbb{R}^N)$ satisfying (\ref{eq01})-(\ref{eq02}) is usually referred as a normalized solution of (\ref{eq01}). The study of normalized solutions for (\ref{eq01}) is motivated by the search of
standing waves solutions with form $\Psi(t,x)=e^{-i\lambda t}u(x)$ of prescribed mass $c$ for the following time-dependent Schr\"odinger equation
\begin{equation*}
i\Psi_t(t,x)+\Delta_{x}\Psi(t,x)+g\left(|\Psi(t,x)|^2\right)\Psi(t,x)=0, \quad  (t,x)\in \mathbb{R}\times \mathbb{R}^{N}.
\end{equation*}

Consider the associated energy functional $J: H^1(\mathbb{R}^N)\to \mathbb{R}$ given by
\begin{displaymath}
J(u):=\frac{1}{2}\int_{\mathbb{R}^{N}}|\nabla u|^{2}dx-\int_{\mathbb{R}^N}F(u)dx,
\end{displaymath}
where $F(u):=\int_{0}^{u}f(s)ds$.
Set $\mathcal{S}_{c}:=\Big\{u\in H^{1}(\mathbb{R}^{N}): \int_{\mathbb{R}^N}|u|^2dx=c\Big\}$.
Clearly, $J$ is of $C^1$ under suitable conditions on $f$, and any critical point of $J$ on $\mathcal{S}_c$ corresponds to a normalized solution of (\ref{eq01}). Furthermore, each normalized solution of (\ref{eq01}) stays in the following Nehari-Pohozaev type set
\begin{equation*}
\mathcal{M}:=\Big\{u\in H^1(\mathbb{R}^N)\setminus\{0\}: \mathcal{P}(u)=0\Big\}.
\end{equation*}
Here $\mathcal{P}: H^1(\mathbb{R}^N)\to \mathbb{R}$ is the Nehari-Pohozaev functional defined by
\begin{equation*}
 \mathcal{P}(u):=\int_{\mathbb{R}^N}|\nabla u|^2dx-\frac{N}{2}\int_{\mathbb{R}^N}H(u)dx, \quad u\in H^1(\mathbb{R}^N),
\end{equation*}
where $H(u):=f(u)u-2F(u)$.

If $f$ admits a $L^2$ subcritical growth at infinity, i.e., $f(s)$ has a growth $|s|^{p-1}$ with $p<2+\frac{4}{N}$ as $|s|\to+\infty$, then $J|_{\mathcal{S}_{c}}$ is bounded below and one can use the minimization method to get a global minimizer, see for example \cite{Caze82,Shib2014}.

If $f$ admits a $L^2$ supercritical growth at infinity, i.e., $p>2+\frac{4}{N}$, then $J|_{\mathcal{S}_{c}}$ is unbounded below and the direct minimization does not work. The first breakthrough in the case of $L^2$ supercritical was made by Jeanjean \cite{Jean97}, where a mountain-pass type argument for the scaled functional $\tilde{J}(u,t):=J(t\star u)$ with $t\star u(\cdot):= t^{\frac{N}{2}}u(t\cdot)$ was introduced. Subsequently, Bartsch and de Valerioda \cite{Bart2013} applied the genus theory to obtain infinitely many normalized solutions of (\ref{eq01}). Ikoma and Tanaka \cite{Ikom2019} established a deformation result on $\mathcal{S}_c$, and gave an alternative proof of the results in \cite{Bart2013,Jean97}.
Bartsch and Soave \cite{Bart2017,BS2017-2} demonstrated that the set $\mathcal{S}_{c}\cap\mathcal{M}$ is a $C^1$ manifold and constitutes a natural constraint. They developed a minimax approach on $\mathcal{S}_{c}\cap\mathcal{M}$ based on the $\sigma$-homotopy stable family of compact subsets of $\mathcal{S}_{c}\cap\mathcal{M}$ to investigate the existence and multiplicity of normalized solutions of (\ref{eq01}). In these studies, the following Ambrosetti-Rabinowitz (AR) condition:
\begin{itemize}
\item [$(A)$] ~~~~~~\mbox{there exists}~~$\alpha>2+\frac{4}{N}$~~\mbox{such that}~~$f(s)s\ge \alpha F(s)>0, \quad  s\neq0$
\end{itemize}
plays an essential role in obtaining the bounded constrained Palais-Smale sequences of $J$.

Jeanjean and Lu \cite{Jean(1)2020} investigated the normalized solutions of (\ref{eq01}) when $f$ satisfies a certain monotonicity condition (see $(f_3)$ below), but not the (AR) condition. For any $c>0$, define
$$
m(c):=\inf\limits_{u\in\mathcal{S}_{c}\cap\mathcal{M}}J(u).
$$
By analysing the variation of ground state energy $m(c)$ with respect to the prescribed mass $c$, they modified the minimax arguments \cite{Bart2017} to establish the existence of normalized ground state solutions.


 Bieganowski and Mederski \cite{Bieg2021} introduced a constrained minimization method for the $L^2$ supercritical problem without imposing the (AR) condition. In order to establish the existence of normalized ground state solutions, they first pursued the existence of minimizers for $J$ on $\mathcal{D}_{c}\cap \mathcal{M}$, where
\begin{eqnarray*}
\mathcal{D}_{c}:=\left\{u\in H^{1}(\mathbb{R}^{N})\setminus\{0\}: \int_{\mathbb{R}^{N}}|u|^{2}dx\leq c\right\}.
\end{eqnarray*}
The proof was then performed through an analysis of Lagrange multipliers for the constraints $\mathcal{S}_c$ and $\mathcal{M}$, respectively.

Both of \cite{Jean(1)2020} and \cite{Bieg2021} were carried out under the assumption that the nonlinearity $f$ satisfies $L^2$-supercritical but Sobolev subcritical growth at infinity. In the case where $f$ admits Sobolev critical growth, Soave \cite{Soave(2)2020} initially explored the following Sobolev critical Schr\"{o}dinger equation
\begin{equation}\label{1.4}
-\Delta u=\lambda u+\mu|u|^{q-2}u+|u|^{2^{*}-2}u \qquad \mbox{in} ~  \mathbb{R}^{N}, ~ N\geq 3,
\end{equation}
where $2<q<2^{*}:=\frac{2N}{N-2}$ and the condition
$\mu c^{(1-\gamma_{q})q}<\alpha(N, q)$ is imposed. Here $\gamma_q=\frac{N(q-2)}{2q}$, $\alpha(N, q)=+\infty$ if $N=3,4$ and $\alpha(N, q)$ is finite for $N\ge5$. Soave \cite{Soave(2)2020} employed the Ekeland variational principle and Schwarz rearrangement techniques to establish the existence of normalized ground state solutions, acting as local minimizers, specifically for the range $2<q<2+\frac{4}{N}$. In scenarios where $q$ exhibits $L^2$-critical and $L^2$-supercritical growth, i.e., $2+\frac{4}{N}\le q<2^*$, the Ghoussoub minimax principle was utilized to demonstrate the existence of mountain pass type normalized ground state solutions.
Subsequently,
under the perturbation term $\mu|u|^{q-2}u$, exhibiting $L^2$-subcritical, $L^2$-critical, and $L^2$-supercritical characteristics respectively, significant progress has been made towards a complete understanding of (\ref{1.4}). Interested readers are referred to \cite{Alves2022, Jean(2)2020, Jean2021, Li2021, Mede2022, Wei2021}, and the references therein for further details.

In this paper, we develop a constraint minimization approach to study the existence of normalized ground state solutions of (\ref{eq01}) when $f$ satisfies a general Sobolev critical growth. Using a monotonicity condition as in \cite{Jean(1)2020} but not the (AR) condition, we shall establish the existence of normalized ground state solutions of (\ref{eq01}) for all $c>0$.

Our assumptions are formulated as follows:
 \begin{itemize}
\item [($f_1$)]~$ f\in C^{1}(\mathbb{R}, \mathbb{R})$ and $\lim\limits_{|s|\rightarrow 0}\frac{F(s)}{|s|^{2+\frac{4}{N}}}=0;$
\item [($f_2$)]~there exists $0<\eta<+\infty$ such that $\eta:=\lim\limits_{|s|\rightarrow \infty}\frac{F(s)}{|s|^{2^{*}}};$
\item [($f_3$)]~$\frac{H(s)}{|s|^{2+\frac{4}{N}}}$ is strictly decreasing on $(-\infty,0)$ and strictly increasing on $(0,+\infty)$;
\item [($f_4$)]~there exist constants $p>2+\frac{4}{N}$ and $\mu>0$ such that
$$sgn(s)f(s)\geq \mu |s|^{p-1}, \textcolor{red}{\quad s\in\mathbb{R}\setminus\{0\},}$$
where $ sgn: \mathbb{R}\setminus \{0\}\to \mathbb{R}$ is defined by
    \begin{displaymath}
    sgn(s)=
    \left\{
    \begin{aligned}
    &1, &s>0,\\
    &-1, &s<0;
    \end{aligned}
    \right.
    \end{displaymath}
\item [($f_5$)]~$f(s)s< 2^{*}F(s), \quad \forall s\in \mathbb{R}\setminus \{0\}.$
\end{itemize}
The main result of this paper is stated as follows.
\renewcommand\thetheorem{1.\arabic{theorem}}
\begin{theorem}\label{Thm1}
Assume that $(f_1)$-$(f_5)$ hold. Then, for any $c>0$, there exists $\mu_{0}>0$ such that problem (\ref{eq01}) admits a normalized ground state solution pair $(u, \lambda)\in H^1(\mathbb{R}^N)\times \mathbb{R}$ with $\lambda<0$ for all $\mu>\mu_{0}$.
\end{theorem}
\begin{remark}
Condition $(f_4)$ was utilized to investigate normalized ground state solutions for problem (\ref{eq01}) with $N=2$ when the nonlinearity $f$ displays exponential critical growth at infinity \cite{Alves2022,Chang23}. By imposing more specific assumptions and performing a detailed energy estimate, it becomes possible to eliminate the requirement that $\mu$ is large. Additional insights into this can be found in \cite{Chen23}.
\end{remark}
\begin{remark}
 Following the strategy \cite{Bieg2021}, Mederski and Schino \cite{Mede2022} obtained the existence of ground state solutions of problem (\ref{eq01}) when $c>c^*$ for some $c_*>0$, where the nonlinearity is of the following form
\begin{eqnarray*}\label{2-19-1}
f(u)=g(u)+|u|^{2^*-2}u.
\end{eqnarray*}
The lower power nonlinearity $g$ was required to satisfy several assumptions distinct from $(f_3)$ but not the (AR) type condition. The approach \cite{Mede2022} addressing the lack of compactness due to critical growth relies on a profile decomposition theorem established in \cite{Mede2020}.
However, in our situation, we cannot directly apply the profile decomposition theorem as done in \cite{Bieg2021,Mede2022}.
\end{remark}


To establish Theorem \ref{Thm1}, we employ a direct minimization argument for $J$ on $\mathcal{S}_c\cap \mathcal{M}$. Our work space is $H^1(\mathbb{R}^N)$, necessitating the addressing of challenges arising from the absence of a compact embedding $H^1(\mathbb{R}^N)\subseteq L^r(\mathbb{R}^N)$ for $r\in [2,2^*]$.
Our strategy for restoring compactness involves a detailed examination of the monotonicity of the ground state energy $m(c)$ concerning the mass parameter $c$, as elaborated in \cite{Jean(1)2020, Yang2020}. This analysis is coupled with a Brezis-Lieb type result for Sobolev critical nonlinearity and an inequality relationship between the energy functional $J$ and the Nehari-Pohozaev functional $\mathcal{P}$. The presence of Sobolev critical growth in the nonlinearity adds intricacy to the problem.

Here is a brief outline for the proof of Theorem \ref{Thm1}: Start with a minimizing sequence ${u_n}\subset \mathcal{S}_c\cap \mathcal{M}$ of $J$ at $m(c)$. Using conditions $(f_3)-(f_4)$, we apply concentration compactness arguments (refer to \cite{Lions84}) and energy estimates to establish the boundedness of $\{u_n\}$. Subsequently, assuming $u_n\rightharpoonup u_0$ in $H^{1}(\mathbb{R}^{N})$,
by the Lions lemma, we deduce $\|u_0\|_{2}^{2}:=c_0\in(0,c].$
Then, we can skillfully utilize a Brezis-Lieb splitting argument and $(f_3)$ to prove that $\mathcal{P}(u_0)\leq 0$, which implies that there exists $t_0\in(0,1]$ such that $t_0\star u_0\in \mathcal{S}_{c_0}\cap \mathcal{M},$
\begin{equation*}
J(u_0)-J(t_0\star u_0)\geq \frac{1-t_0^{2}}{2}\mathcal{P}(u_0).
\end{equation*}
Hence by delicate integral estimates, we have
\begin{eqnarray*}
m(c) \geq J(u_0)-\frac{1}{2}\mathcal{P}(u_0)
\geq J(t_0\star u_0)-\frac{t_0^{2}}{2}\mathcal{P}(u_0)\geq J(t_0\star u_0)\geq m(c_0).
\end{eqnarray*}
On the other hand, using the monotonicity of $m(c)$ with respect to $c>0$, we infer that $m(c)\leq m(c_0),$ which implies that $\mathcal{P}(u_0)=0$ and thus $m(c)=m(c_0).$ Hence, we get $J(u_0)=m(c_0).$ Since the function $c\mapsto m(c)$ is strictly decreasing at $c_0$, we derive $c_0=c$ and then $u_0\in\mathcal{S}_{c}\cap \mathcal{M}.$  Thereafter, by showing that $\mathcal{M}$ is a natural constraint, we prove that the minimizer $u_0$ of $J$ on $\mathcal{S}_c\cap\mathcal{M}$ is a normalized ground state solution of (\ref{eq01}).

The paper is organized as follows. In Section $2,$ we introduce some preliminary results. In Section $3$, the monotonicity of $m(c)$ with respect to $c>0$ will be analyzed. In Section $4$, we give the proof of Theorem 1.1.

Regarding the notations, for $p\geq 1,$ the (standard) $L^{p}-$ norm of $u\in L^{p}(\mathbb{R}^{N})$ is denoted by $\|u\|_{p}$ and $\|\cdot\|$ denotes the norm in $H^{1}(\mathbb{R}^{N}).$

\section{Preliminaries}\label{Preli}
\setcounter{equation}{0}
\renewcommand\theequation{2.\arabic{equation}}
In this section, we prove some preliminary results. We recall that, for $N\ge2$ and $q\in (2, 2^*)$, there exists $C_{N,q}>0$ depending on $N$ and $q$ such that the following Gagliardo-Nirenberg inequality holds:
\begin{equation}\label{2.14}
\|u\|_{q}\leq C_{N,q}\|u\|^{1-\gamma_{q}}_{2}\|\nabla u\|^{\gamma_{q}}_{2}, \quad \forall u\in H^{1}(\mathbb{R}^{N}),
\end{equation}
where $\gamma_{q}:=N(\frac{1}{2}-\frac{1}{q}).$ See \cite{W1983}. If $N\ge3$, by \cite{W1996}, there exists an optimal constant $S>0$ depending only on $N$ such that
\begin{equation}\label{2.15}
S\|u\|^{2}_{2^{*}}\leq\|\nabla u\|^{2}_{2}, \quad \forall u\in H^{1}(\mathbb{R}^{N}).
\end{equation}

\renewcommand\thetheorem{2.\arabic{theorem}}
\begin{lemma}\label{lem2} Assume that $(f_{1})-(f_{2})$ and $(f_{4})-(f_{5})$ hold. Then
$\inf\limits_{\mathcal{S}_{c}\cap\mathcal{M}}\|\nabla u\|_{2}>0.$
\end{lemma}
\begin{proof}
We assume by contradiction that there exists $\{u_{n}\}\subset \mathcal{S}_{c}\cap \mathcal{M}$ such that
 \begin{equation}\label{8-24-1}
 \|\nabla u_{n}\|_{2}\rightarrow 0.
 \end{equation}
 By $(f_{1})-(f_{2})$ and $(f_{4})$, for any $\epsilon>0,$ there exist $p<p_1<2^{*}$ and $C_{\epsilon,p_1}>0$ such that
\begin{equation}\label{2.0}
F(s)\leq C_{\epsilon,p_1}|s|^{p_1}+(\eta+\epsilon)|s|^{2^{*}}, \quad \forall s\in \mathbb{R}.
\end{equation}
Then, by (\ref{2.14})-(\ref{2.0}), $(f_{5})$ and $\gamma_{p_{1}}p_{1}>2,$ we get
\begin{eqnarray*}
 \int_{\mathbb{R}^{N}}f(u_{n})u_{n}dx
 &\leq&  2^{*}\int_{\mathbb{R}^{N}}F(u_{n})dx  \nonumber\\
&\leq& 2^{*}C_{\epsilon,p_1}C_{N, p_{1}}^{p_{1}}c^{\frac{1-\gamma_{p_{1}}}{2}p_{1}}\|\nabla u_{n}\|_{2}^{\gamma_{p_{1}}p_{1}-2}\int_{\mathbb{R}^{N}}|\nabla u_{n}|^{2}dx\nonumber\\
&~~~~~~~~~~~~~~~~~~~~~&+2^{*}(\eta+\epsilon) S^{-\frac{2^{*}}{2}}\|\nabla u_{n}\|_{2}^{2^{*}-2}\int_{\mathbb{R}^{N}}|\nabla u_{n}|^{2}dx \nonumber\\
 &\leq& \frac{1}{4}\int_{\mathbb{R}^{N}}|\nabla u_{n}|^{2}dx
\end{eqnarray*}
for $n$ large enough.
Hence, using $(f_{5})$ again,
\begin{eqnarray*}
\int_{\mathbb{R}^{N}}H(u_{n})dx
&=&\int_{\mathbb{R}^{N}}\big[f(u_{n})u_{n}-2F(u_{n})\big]dx\nonumber\\
&\leq& \frac{2}{N}\int_{\mathbb{R}^{N}}f(u_{n})u_{n}
\leq \frac{1}{N}\int_{\mathbb{R}^{N}}|\nabla u_{n}|^{2}dx,
\end{eqnarray*}
which together with $\{u_{n}\}\subset \mathcal{M}$ implies that
\begin{displaymath}
0=\mathcal{P}(u_{n})=\int_{\mathbb{R}^{N}}|\nabla u_{n}|^{2}dx-\frac{N}{2}\int_{\mathbb{R}^{N}}H(u_{n})dx\geq\frac{1}{2}\int_{\mathbb{R}^{N}}|\nabla u_{n}|^{2}dx.
\end{displaymath}
This is a contradiction with $\{u_{n}\}\subset \mathcal{S}_{c}$. Therefore, $\inf\limits_{\mathcal{S}_{c}\cap\mathcal{M}}\|\nabla u\|_{2}>0.$
\end{proof}

For any $u\in H^1(\mathbb{R}^N)\setminus \{0\}$ and $t>0$, set
\begin{equation*}
(t\star u)(x):=t^{\frac{N}{2}}u(tx),~~~\forall x\in \mathbb{R}^N.
\end{equation*}
Clearly, $t\star u\in \mathcal{S}_c$ if $u\in \mathcal{S}_c$.
\begin{lemma}\label{lem1} Suppose that $(f_{1})-(f_{4})$ hold. Then, for any $u\in H^1(\mathbb{R}^N)\setminus \{0\}$, there exists a unique number $t_{u}>0$ such that $t_{u}\star u\in \mathcal{M}.$ Moreover, $J(t_{u}\star u)>J(t\star u)$ for all $t>0$ with $t\neq t_u.$
\end{lemma}
\begin{proof}
Using similar argument as in Lemma \ref{lem2}, by $(f_{1})-(f_{2}), (f_4)$, for any $\epsilon>0,$ there exist $p<p_{1}<2^{*}, C_{\epsilon,p_{1}}>0$ such that
\begin{eqnarray*}
 J(t\star u)
\ge \frac{t^{2}}{2}\int_{\mathbb{R}^{N}}|\nabla u|^{2}dx-C_{\epsilon,p_1}t^{\frac{N}{2}(p_{1}-2)}\int_{\mathbb{R}^{N}}|u|^{p_{1}}dx-(\eta+\epsilon) t^{2^{*}}\int_{\mathbb{R}^{N}} |u|^{2^{*}}dx
\end{eqnarray*}
and
$$
J(t\star u)\leq \frac{t^2}{2}\int_{\mathbb{R}^{N}}|\nabla u|^{2}dx-\frac{\mu}{p}t^{\frac{N}{2}(p-2)}\int_{ \mathbb{R}^{N}}|u|^{p}dx.
$$
Then, by $p_{1}>p>2+\frac{4}{N},$ we have
\begin{align*}
&J(t\star u)\rightarrow 0^{+}\quad\quad\ \ \ \  \mbox{as}\quad t\rightarrow 0^{+},\\
&J(t\star u)\rightarrow-\infty \quad\quad\ \  \mbox{as}\quad t\rightarrow +\infty,
\end{align*}
which imply that $\Phi_u(t):=J(t\star u)$ admits a global maximum point $t_{u}>0$ and thus
$\frac{d}{dt}\Phi_u(t)|_{t=t_{u}}=0.$
Since
\begin{equation*}
\frac{d}{dt}J(t\star u)=t^{-1}\mathcal{P}(t\star u),
\end{equation*}
we deduce $t_{u}\star u\in\mathcal{M}.$

In what follows, we prove the uniqueness. Assume by contradiction that there exist $0<\tilde{t}_{u}<t_u$ such that $\tilde{t}_{u}\star u,  t_{u}\star u\in \mathcal{M}$. Then
\begin{equation}\label{8-25-1}
t_u^2\int_{\mathbb{R}^{N}}|\nabla  u|^{2}dx-\frac{N}{2}t_u^{-N}\int_{\mathbb{R}^{N}}H(t_{u}^{\frac{N}{2}}u)dx
=0
\end{equation}
and
\begin{equation}\label{2-1}
 \tilde{t}_u^2\int_{\mathbb{R}^{N}}|\nabla u|^{2}dx-\frac{N}{2}\tilde{t}_u^{-N}\int_{\mathbb{R}^{N}}H(\tilde{t}_{u}^{\frac{N}{2}}u)dx=0.
\end{equation}
By similar arguments as in Remark 2.2 in \cite{Jean(1)2020}, we may assume that $\frac{H(s)}{|s|^{2+\frac{4}{N}}}$ is continuous on $\mathbb{R}$, $\frac{H(s)}{|s|^{2+\frac{4}{N}}}=0$ at $s=0$. Furthermore, the ratio is strictly decreasing on $(-\infty, 0]$ and strictly increasing on $[0, +\infty)$. Then by (\ref{8-25-1})-(\ref{2-1}) we get
\begin{displaymath}
\int_{\mathbb{R}^{N}}\left(\frac{H(t_{u}^{\frac{N}{2}}u)}{|t_{u}^{\frac{N}{2}}u|^{2+\frac{4}{N}}}
-\frac{H(\tilde{t}_{u}^{\frac{N}{2}}u)}{|\tilde{t}_{u}^{\frac{N}{2}}u|^{2+\frac{4}{N}}}\right)|u|^{2+\frac{4}{N}}dx=0.
\end{displaymath}
However, by $(f_{3})$ we can deduce
\begin{displaymath}
\int_{\mathbb{R}^{N}}\left(\frac{H(t_{u}^{\frac{N}{2}}u)}{|t_{u}^{\frac{N}{2}}u|^{2+\frac{4}{N}}}
-\frac{H(\tilde{t}_{u}^{\frac{N}{2}}u)}{|\tilde{t}_{u}^{\frac{N}{2}}u|^{2+\frac{4}{N}}}\right)|u|^{2+\frac{4}{N}}dx>0,
\end{displaymath}
which provides a contradiction. Hence $\tilde{t}_{u}=t_{u}$, which implies that $t_{u}\star u\in \mathcal{M}$ is the unique global maximum point of $\Phi_u$ and $J(t_{u}\star u)>J(t\star u)$ for all $t\neq t_u.$
\end{proof}
\begin{lemma}\label{lem4} Assume that $(f_{1})-(f_{5})$ hold. Then $\inf\limits_{\mathcal{S}_{c}\cap\mathcal{M}}J(u)=\inf\limits_{u\in\mathcal{S}_{c}}\max\limits_{t>0}J(t\star u)>0.$
\end{lemma}
\begin{proof} By Lemma \ref{lem1}, it is easily seen that the $"="$ holds. Then, by $(f_{1})-(f_{2}),(f_{4})$ and similar arguments as in Lemma \ref{lem2}, for any $\epsilon>0,$ there exist $p<p_{1}<2^{*}, C_{\epsilon,p_{1}}>0$ such that for any $t>0$ and $u\in \mathcal{S}_{c}\cap\mathcal{M}$, we have
\begin{align*}
&J(u)\nonumber\\
&\geq \frac{t^{2}}{2}\int_{\mathbb{R}^{N}}|\nabla u|^{2}dx-\int_{\mathbb{R}^{N}}\Big(C_{\epsilon,p_1}t^{\frac{N}{2}(p_{1}-2)}|u|^{p_{1}}+(\eta+\epsilon)t^{2^{*}}|u|^{2^{*}}\Big)dx \nonumber\\
&\geq \frac{t^{2}}{2}\|\nabla u\|^{2}_{2}- \Big(C_{\epsilon,p_1}C_{N, p_{1}}^{p_{1}}t^{\frac{N(p_{1}-2)}{2}}c^{\frac{1-\gamma_{p_{1}}}{2}p_{1}}\|\nabla u\|^{\frac{N(p_{1}-2)}{2}}_{2}+(\eta+\epsilon)S^{-\frac{2^{*}}{2}}t^{2^{*}}\|\nabla u\|^{2^{*}}_{2}\Big).
\end{align*}
Hence, in view of $p_1>2+\frac{4}{N}$, taking $t=\delta/\|\nabla u\|_{2}$ with $\delta>0$ small enough, we deduce
$$
J(u)\geq \frac{\delta^{2}}{2}-\left(C_{\epsilon,p_1}C_{N, p_{1}}^{p_{1}}c^{\frac{1-\gamma_{p_{1}}}{2}p_{1}}\delta^{\gamma_{p_{1}}p_{1}}+(\eta+\epsilon) S^{-\frac{2^{*}}{2}}\delta^{2^{*}}\right)\ge \frac{\delta^{2}}{4}>0.
$$
\end{proof}
\begin{lemma}\label{lem3} Assume that $(f_{1})-(f_{4})$ hold. Then, for any $u\in H^{1}(\mathbb{R}^{N})\setminus \{0\},$ we have
\begin{description}
\item[(i)] The map $u\mapsto t_{u}$ is continuous.
\item[(ii)] $t_{u(\cdot+y)}=t_{u}$ for any $y\in\mathbb{R}^{N}.$
\end{description}
\end{lemma}
\begin{proof}
For (i), by Lemma \ref{lem1}, the mapping $u\mapsto t_{u}$ is well defined. Let $\{u_{n}\}\subset H^{1}(\mathbb{R}^{N})\setminus \{0\}$ be any sequence such that $u_{n}\rightarrow u$ in $H^{1}(\mathbb{R}^{N})$. We first show that $\{t_{u_{n}}\}$ is bounded. If not, $t_{u_{n}}\rightarrow +\infty.$
Then, by $(f_{4})$ it follows that
\begin{eqnarray*}
0\leq t_{u_{n}}^{-2}J(t_{u_{n}}\star u_{n})
& = &\frac{1}{2}\int_{\mathbb{R}^{N}}|\nabla u_{n}|^{2}dx-t_{u_{n}}^{-(2+N)}\int_{\mathbb{R}^{N}}F(t_{u_{n}}^{\frac{N}{2}}u_{n})dx \nonumber\\
&\leq & \frac{1}{2}\int_{\mathbb{R}^{N}}|\nabla u_{n}|^{2}dx-\frac{\mu}{p}t_{u_{n}}^{\frac{N}{2}\left(p-(2+\frac{4}{N})\right)}\int_{\mathbb{R}^{N}}|u_{n}|^{p}dx \nonumber\\
&\to & -\infty \quad \mbox{as}~~t\to+\infty,
\end{eqnarray*}
which is a contradiction. Hence, the sequence $\{t_{u_{n}}\}$ is bounded, which implies that there exists $t^{*}\ge0$ such that
$t_{u_{n}}\rightarrow t^{*}$.
Due to $u_{n}\rightarrow u$ in $H^{1}(\mathbb{R}^{N})$ and $\mathcal{P}(t_{u_{n}}\star u_{n})=0,$ we get $t_{u_{n}}\star u_{n}\rightarrow t^{*}\star u$ in $H^{1}(\mathbb{R}^{N})$ and $\mathcal{P}(t^{*}\star u)=0.$ By Lemma \ref{lem1}, we have $t^{*}=t_{u_{n}}$ and thus (i) is proved. By the definition of $t_{u(\cdot+y)}$ and direct computations, we get (ii).
\end{proof}

%
%


%


\section{Monotonicity of the ground state energy}\label{ground state energy}
\setcounter{equation}{0}
\renewcommand\theequation{3.\arabic{equation}}

In this section, we shall study the behavior of the ground state energy $m(c)$.

\renewcommand\thetheorem{3.\arabic{theorem}}
\begin{lemma}\label{conti}
Assume that $(f_{1})-(f_{5})$ hold. Then the function $c\mapsto m(c)$ is continuous.
\end{lemma}
\begin{proof}
For any $c>0,$ we take $\{c_n\}$  be such that $c_n\rightarrow c$ as $n\rightarrow \infty.$
For any $u\in\mathcal{S}_{c}\cap\mathcal{M},$ set
$u_n=\sqrt{\frac{c_n}{c}}u, \ n\in\mathbb{N}^{+}.$
Clearly, $u_n\in S_{c_n}$ and $u_n\rightarrow u$ in $H^{1}(\mathbb{R}^{N}).$ Then, by Lemma \ref{lem3} (i), we have
\begin{equation*}
t_{u_n}\to t_{u}=1,~~~t_{u_n}\star u_n\rightarrow t_{u}\star u=u \quad \mbox{in}\  H^{1}(\mathbb{R}^{N}),
\end{equation*}
which implies that $$\lim\limits_{n\rightarrow \infty}\sup m(c_n)\leq \lim\limits_{n\rightarrow \infty}\sup J(t_{u_n}\star u_n)=J(u).$$
Since $u\in\mathcal{S}_{c}\cap\mathcal{M}$ is arbitrary, we get
\begin{equation}\label{2.7}
\lim\limits_{n\rightarrow \infty}\sup m(c_n)\leq m(c).
\end{equation}

In what follows, we prove
\begin{equation}\label{2.8}
m(c)\leq \lim\limits_{n\rightarrow \infty}\inf m(c_n).
\end{equation}
There exists $v_n\in S_{c_n}\cap \mathcal{M}$ such that
\begin{equation}\label{2.9}
J(v_n)\leq m(c_n)+\frac{1}{n}.
\end{equation}
We claim $\{v_n\}$ is bounded in $H^{1}(\mathbb{R}^{N}).$ If not, due to $\|v_n\|_{2}^{2}=c_n\rightarrow c,$ we deduce $\|\nabla v_n\|_{2}\rightarrow \infty.$
Set $w_n=\frac{1}{s_n}\star v_n,$ where $s_n=\|\nabla v_n\|_{2}.$ Clearly,
$$
s_n\rightarrow \infty, \quad  s_n\star w_n\in\mathcal{S}_{c_n}\cap \mathcal{M} \quad  \mbox{and}  \quad \|\nabla w_n\|_{2}=1,
$$
which imply that $\{w_n\}$ is bounded in $H^{1}(\mathbb{R}^{N}).$ Set
$$\hat{\delta}:=\lim\limits_{n\rightarrow \infty}\sup\Big(\sup\limits_{y\in \mathbb{R}^{N}}\int_{B(y,1)}|w_n|^{2}dx\Big).$$
 We distinguish the following two cases:

\vspace*{4pt}\noindent\textbf{Case 1.} $\hat{\delta}>0$, i.e., non-vanishing occurs. Then, up to a subsequence, there exists $\{z_n\}\subset \mathbb{R}^{N}$ such that
\begin{displaymath}
w_n(\cdot+z_n)\rightharpoonup w\neq 0 \quad \mbox{in}\  H^{1}(\mathbb{R}^{N}) \quad \mbox{and} \quad  w_n(x+z_n)\rightarrow w(x) \quad \mbox{for}\  a.e. \ x\in \mathbb{R}^{N}.
\end{displaymath}
By $(f_{4})$ and Lemma \ref{lem4}, we have
\begin{equation*}
\begin{aligned}
 0\leq s_n^{-2}J(v_n)
\leq\frac{1}{2}-\frac{\mu}{p}s_n^{\frac{N}{2}\left(p-(2+\frac{4}{N})\right)}\int_{\mathbb{R}^{N}}\big|w_n(x+z_n)\big|^{p}
dx
\rightarrow -\infty,
\end{aligned}
\end{equation*}
which is a contradiction.

\vspace*{4pt}\noindent\textbf{Case 2.} $\hat{\delta}=0$, i.e., $\{w_n\}$ is vanishing.
By the Lions lemma \cite{Lions84}, we deduce
\begin{equation}\label{3.5}
\int_{\mathbb{R}^{N}}|w_{n}|^{p_{1}}dx\rightarrow 0.
\end{equation}
Moreover, by $(f_{1})-(f_{2}),(f_{4}),$ (\ref{2.15}), (\ref{2.9}), Lemma \ref{lem1} and $\|\nabla w_{n}\|_{2}=1,$ for any $\epsilon>0,$ there exist $p<p_{1}<2^{*}, C_{\epsilon,p_{1}}>0$ such that for any $s>0,$ we have
\begin{eqnarray*}
 m(c_n)+\frac{1}{n}
&\geq & \frac{1}{2}s^{2}\int_{\mathbb{R}^{N}}|\nabla w_n|^{2}dx-s^{-N}\int_{\mathbb{R}^{N}}F(s^{\frac{N}{2}}w_n)dx\nonumber\\
&\geq& \frac{1}{2}s^{2}-s^{\frac{N}{2}(p_{1}-2)}C_{\epsilon,p_1}\int_{\mathbb{R}^{N}}|w_n|^{p_{1}}dx-s^{2^{*}}(\eta+\epsilon) S^{\frac{N}{2-N}}.
\end{eqnarray*}
Hence, letting $n\to\infty,$ by (\ref{2.7}) and (\ref{3.5}), we obtain $$m(c)\geq \frac{1}{2}s^{2}-s^{2^{*}}(\eta+\epsilon) S^{\frac{N}{2-N}}, \quad\forall s\in \mathbb{R}\setminus \{0\}.$$
Taking $s=\left(2^{*}(\eta+\epsilon)\right)^{\frac{2-N}{4}}S^{\frac{N}{4}},$ we get
\begin{equation}\label{4.4}
m(c)\geq \frac{1}{N}\left(2^{*}(\eta+\epsilon)\right)^{\frac{2-N}{2}}S^{\frac{N}{2}}.
\end{equation}
However, by $(f_{4}),$ for any $u\in\mathcal{S}_{c}$,
\begin{eqnarray}\label{2.21}
m(c)\leq \max\limits_{t\geq 0}J(t\star u)&\leq& \max\limits_{t\geq 0}\Big\{\frac{t^{2}}{2}\int_{\mathbb{R}^{N}}|\nabla u|^{2}dx-\frac{\mu}{p}t^{\frac{N(p-2)}{2}}\int_{\mathbb{R}^{N}}|u|^{p}dx\Big\}\nonumber\\
&\le&C\left(\frac{1}{\mu}\right)^{\frac{4}{N(p-2)-4}}\rightarrow 0 \quad as \ \mu\rightarrow +\infty,
\end{eqnarray}
which contradicts with $(\ref{4.4}).$ Hence, $\{v_n\}$ is bounded in $H^{1}(\mathbb{R}^{N}).$

Take $\tilde{v}_n:=v_n(\frac{\cdot}{\tau_n})$
with $\tau_n=(\frac{c}{c_n})^{\frac{1}{N}}.$ Clearly, $\tilde{v}_n\in \mathcal{S}_{c},$ $t_{\tilde{v}_n}\star\tilde{v}_n\in \mathcal{M}$ and $\{\tilde{v}_n\}$ is bounded in $H^{1}(\mathbb{R}^{N}).$
We claim that
\begin{equation}\label{2.2}
\lim\limits_{n\rightarrow \infty}\sup t_{\tilde{v}_n}<+\infty.
\end{equation}
We first show $\{\tilde{v}_n\}$ is non-vanishing, i.e.,
\begin{displaymath}
\tilde{\delta}:=\lim\limits_{n\rightarrow \infty}\sup\Big(\sup\limits_{y\in\mathbb{R}^{N}}\int_{B(y,1)}|\tilde{v}_n|^{2}dx\Big)>0.
\end{displaymath}
If not, $\tilde{\delta}=0,$ then by the Lions lemma \cite{Lions84}, $\tilde{v}_n\rightarrow 0$ in $L^{q}(\mathbb{R}^{N})$ for $q\in (2, 2^{*}).$ Clearly,
\begin{equation}\label{3-1}
\int_{\mathbb{R}^{N}}|v_n|^{p_1}dx=\tau^{-N}_n\int_{\mathbb{R}^{N}}|\tilde{v}_n|^{p_1}dx\rightarrow 0.
\end{equation}
Then, by $(f_{1})-(f_{2}), (f_{4})-(f_{5}),$ $\mathcal{P}(v_n)=0$ and $(\ref{2.15})$, for any $\epsilon>0,$ there exist $ p<p_{1}<2^{*},C_{\epsilon,p_{1}}>0$ such that
\begin{eqnarray}\label{2.3}
\int_{\mathbb{R}^{N}}|\nabla v_n|^{2}dx
&\leq&\frac{N}{2}\left(2^{*}-2\right)\int_{\mathbb{R}^{N}}F(v_n)dx\nonumber\\
&\leq& 2^{*}C_{\epsilon,p_1}\int_{\mathbb{R}^{N}}|v_n|^{p_{1}}dx+2^{*}(\eta+\epsilon) S^{\frac{N}{2-N}}\Big(\int_{\mathbb{R}^{N}}|\nabla v_n|^{2}dx\Big)^{\frac{N}{N-2}}\nonumber\\
&=&2^{*}(\eta+\epsilon)S^{\frac{N}{2-N}}\Big(\int_{\mathbb{R}^{N}}|\nabla v_n|^{2}dx\Big)^{\frac{N}{N-2}}+o_n(1).
\end{eqnarray}
Since $\{v_n\}$ is bounded in $H^{1}(\mathbb{R}^{N})$, in view of Lemma \ref{lem2}, we may assume that, up to a subsequence, $\|\nabla v_n\|^{2}_{2}\rightarrow l>0.$ Combining with (\ref{2.3}), 
we deduce
\begin{equation}\label{8-26-1}
l\geq\left(2^{*}(\eta+\epsilon)\right)^{\frac{2-N}{2}}S^{\frac{N}{2}}.
\end{equation}
By $(f_{1})-(f_{2}), (f_{4}),$ (\ref{2.15}), (\ref{2.9}), (\ref{3-1}) and Lemma \ref{lem1}, for any $\epsilon>0,$ there exist $p<p_{1}<2^{*},C_{\epsilon,p_{1}}>0$ such that for any $t>0,$ we have
\begin{align*}
& m(c_n)+\frac{1}{n} \nonumber\\
& \geq  \frac{t^{2}}{2}\int_{\mathbb{R}^{N}}|\nabla v_n|^{2}dx-t^{\frac{N(p_{1}-2)}{2}}C_{\epsilon,p_1}\int_{\mathbb{R}^{N}}|v_n|^{p_{1}}dx-t^{2^{*}}(\eta+\epsilon)\int_{\mathbb{R}^{N}}|v_n|^{2^{*}}dx \nonumber\\
&\geq  \frac{t^{2}}{2}\int_{\mathbb{R}^{N}}|\nabla v_n|^{2}dx-t^{2^{*}}(\eta+\epsilon) S^{\frac{N}{2-N}}\Big(\int_{\mathbb{R}^{N}}|\nabla v_n|^{2}dx\Big)^{\frac{N}{N-2}}+o_n(1).
\end{align*}
Then, taking $t=\left(2^{*}(\eta+\epsilon)\right)^{\frac{2-N}{4}}S^{\frac{N}{4}}l^{\frac{1}{2}}$ and $n\to\infty,$ together with (\ref{2.7}) and (\ref{8-26-1}) it follows that
$$
m(c)\geq \frac{1}{N}\left(2^{*}(\eta+\epsilon)\right)^{\frac{2-N}{2}}S^{\frac{N}{2}},
$$
which produces a contradiction with (\ref{2.21}).
Hence, $\{\tilde{v}_n\}$ is non-vanishing, which implies that there exists a sequence $\{y_n\}\subset \mathbb{R}^{N}$ and $v\in H^{1}(\mathbb{R}^{N})$ such that, up to a subsequence, $\tilde{v}_n(\cdot+y_n)\rightarrow v\neq 0$ almost everywhere in $\mathbb{R}^{N}.$

Now, we prove (\ref{2.2}).
By contradiction, up to a subsequence
$ t_{\tilde{v}_n}\rightarrow +\infty \ \ as\ n\rightarrow \infty.$
By Lemma \ref{lem3} (ii), we have
$ t_{\tilde{v}_n(\cdot+y_n)}= t_{\tilde{v}_n}\rightarrow +\infty.$
Using $(f_{4})$, we deduce
\begin{displaymath}
\begin{aligned}
0&  \leq t^{-2}_{\tilde{v}_n(\cdot+y_n)}J\left(t_{\tilde{v}_n(\cdot+y_n)}\star \tilde{v}_n(\cdot+y_n)\right) \nonumber\\
&  =\frac{1}{2}\int_{\mathbb{R}^{N}}\big|\nabla \tilde{v}_n(\cdot+y_n)\big|^{2}dx-\frac{\mu}{p}t_{\tilde{v}_n}^{\frac{N}{2}\left(p-(2+\frac{4}{N})\right)}\int_{\mathbb{R}^{N}} \big|\tilde{v}_n(\cdot+y_n)\big|^{p}dx \nonumber\\
&  \rightarrow -\infty,\nonumber
\end{aligned}
\end{displaymath}
which is a contradiction. Then (\ref{2.2}) holds.

 By (\ref{2.2}) and the fact that $\{v_n\}$ is bounded in $H^{1}(\mathbb{R}^{N})$, it follows that $\big\{t_{\tilde{v}_n}\star v_n \big\}$ is bounded in $H^{1}(\mathbb{R}^{N}).$
By $(\ref{2.9})$ and Lemma \ref{lem1}, we obtain
\begin{align*}
& m(c)\nonumber\\
& \leq  J(t_{\tilde{v}_n}\star v_n)+\Big|J(t_{\tilde{v}_n}\star \tilde{v}_n)-J(t_{\tilde{v}_n}\star v_n)\Big| \nonumber\\
&\leq  \frac{1}{2}\left|\tau^{N-2}_n-1\right|\cdot\int_{\mathbb{R}^{N}}\big|\nabla\left(t_{\tilde{v}_n}\star v_n\right)\big|^{2}dx+\left|\tau^{N}_n-1\right|\cdot\int_{\mathbb{R}^{N}}\big|F(t_{\tilde{v}_n}\star v_n)\big|dx\nonumber\\
&\quad +m(c_n)+\epsilon.
\end{align*}
Hence, in view of $\tau_n\rightarrow 1$ and $\epsilon$ is arbitrary, we get $(\ref{2.8}).$ Together with $(\ref{2.7}),$ it follows that $c\mapsto m(c)$ is continuous.
\end{proof}
\begin{lemma}\label{nonincrea}
Assume that $(f_{1})-(f_{5})$ hold. Then $m(c)$ is non-increasing with respect to $c>0$.
\end{lemma}
\begin{proof}
For any $u_{1}\in \mathcal{S}_{c_{1}},$ set $ u_{2}(x)=\theta^{\frac{2-N}{4}}u_{1}\big(\theta^{-\frac{1}{2}} x\big),$
where $\theta=\frac{c_{2}}{c_{1}}>1.$
Clearly, $\|\nabla u_{2}\|^{2}_{2}=\|\nabla u_{1}\|^{2}_{2}$ and $\|u_{2}\|^{2}_{2}=c_{2}.$
For any $s\in\mathbb{R}\setminus \{0\}$, we define $$G_s(\sigma):=F(s)-\sigma^{\frac{N}{2}}F(\sigma^{\frac{N-2}{4}}s), \quad \forall \sigma\ge1.$$ Clearly, $G_s(1)=0.$ By $(f_{5}),$ we get
\begin{displaymath}
G'_s(\sigma)=-\frac{N}{2}\sigma^{\frac{N-2}{2}}\left(F(\sigma^{\frac{2-N}{4}}s)-\frac{N-2}{2N}f(\sigma^{\frac{2-N}{4}}s)\sigma^{\frac{2-N}{4}}s\right)<0.
\end{displaymath} Then we deduce $G_{s}(\theta)<0,$
 which implies that
\begin{equation}\label{8-26-2}
\int_{\mathbb{R}^{N}}F(t_{2}\star u)dx-\theta^{\frac{N}{2}}\int_{\mathbb{R}^{N}}F\left(\theta^{\frac{2-N}{4}}(t_{2}\star u)\right)dx<0.
\end{equation}
Take $t_{2}>0$ such that $t_{2}\star u_{2}\in \mathcal{S}_{c_{2}}\cap\mathcal{M}.$ Then by Lemma \ref{lem1}, Lemma \ref{lem4} and (\ref{8-26-2}), we get
\begin{align*}
&  m(c_{2})\nonumber\\
&\leq  \frac{1}{2}\int_{\mathbb{R}^{N}}\big|\nabla (t_{2}\star u_{1})\big|^{2}dx-\theta^{\frac{N}{2}}\int_{\mathbb{R}^{N}}F\left(\theta^{\frac{2-N}{4}}(t_{2}\star u_{1})\right)dx \nonumber\\
&= \frac{1}{2}\int_{\mathbb{R}^{N}}\big|\nabla (t_{2}\star u_{1})\big|^{2}dx-\int_{\mathbb{R}^{N}}F(t_{2}\star u_{1})dx+\int_{\mathbb{R}^{N}}F(t_{2}\star u_{1})dx\nonumber\\
& \quad -\theta^{\frac{N}{2}}\int_{\mathbb{R}^{N}}F\left(\theta^{\frac{2-N}{4}}(t_{2}\star u_{1})\right)dx \nonumber\\
&\leq  J(t_{2}\star u_{1}),
\end{align*}
which implies that
 $$
 m(c_{2})\leq \inf\limits_{u\in \mathcal{S}_{c_{1}}}\max\limits_{t>0}J(t\star u)= m(c_{1}).
 $$
\end{proof}

To show the function $c \mapsto m(c)$ is strictly decreasing, the following result is crucial.
\begin{lemma}\label{stric}
Suppose that $(f_{1})-(f_{5})$ hold. Assume that there exists $u\in \mathcal{S}_{c}\cap\mathcal{M}$ such that $J(u)=m(c).$ Then $m(c)>m(c')$ for any $c'>c$ close enough to $c.$
\end{lemma}
\begin{proof}
For any $s>0,$ set $\alpha(s):=J\big(t_{w}\star (sw)\big),$ where $w\in\mathcal{S}_{c'}.$
Clearly,
 $$\alpha(s)=\frac{1}{2}s^{2}t_{w}^{2}\int_{\mathbb{R}^{N}}|\nabla w|^{2}dx-t_{w}^{-N}\int_{\mathbb{R}^{N}}F(st_{w}^{\frac{N}{2}}w)dx.$$
Furthermore, we have
\begin{eqnarray}
\frac{d}{ds}\alpha(s)
&=&st_{w}^{2}\int_{\mathbb{R}^{N}}|\nabla w|^{2}dx-t_{w}^{-N}\int_{\mathbb{R}^{N}}f(st_{w}^{\frac{N}{2}}w)t_{w}^{\frac{N}{2}}wdx\nonumber\\
&=&\frac{1}{s}J'\big(t_{w}\star (sw)\big)\big(t_{w}\star (sw)\big),\label{9-2}
\end{eqnarray}
where $J'(u)u$ denotes the unconstrained derivative of $J$.
By $t_{w}\star w\in\mathcal{M}$ and $(f_{5}),$ we get
$$J'\left(t_{w}\star w\right)\left(t_{w}\star w\right)=\frac{N-2}{2}\int_{\mathbb{R}^{N}}\Big[f\left(t_{w}\star w\right)\left(t_{w}\star w\right )-2^{*}F\left(t_{w}\star w\right)\Big]dx<0.$$
Then by (\ref{9-2}), we deduce for a fixed $\delta_1>0$ small enough,
$\frac{d}{ds}\alpha(s)<0, \forall s\in [1-\delta_{1},1).$
Then by the mean value theorem, we obtain
$$\alpha(1)=\alpha(s)+(1-s)\frac{d}{ds}\alpha(s)|_{s=\xi}<\alpha(s),$$
where $1-\delta_1\leq s<\xi<1.$
For any $c'>c$ close enough to $c,$ set $s_1=\sqrt{\frac{c}{c'}}.$
Clearly, $s_1\in[1-\delta_{1},1)$ and there exists $w\in\mathcal{S}_{c'}$ such that $u=s_1w.$
Hence, by Lemma \ref{lem1}, we have
\begin{eqnarray*}
m(c')\leq J\left(t_{w}\star w \right)=\alpha(1)
<\alpha(s_1)
&=&J\big(t_{w}\star(s_1w)\big) \nonumber\\
&=& J(t_{w}\star u)\leq J(t_{u}\star u)=J(u)=m(c).
\end{eqnarray*}
\end{proof}

By Lemma \ref{nonincrea} and Lemma \ref{stric}, we deduce the following result.
\begin{lemma}\label{decrea}
Assume that $(f_{1})-(f_{5})$ hold. If there exists $u\in \mathcal{S}_{c}\cap\mathcal{M}$ such that $J(u)=m(c),$ then $m(c)>m(c')$ for any $c'>c.$
\end{lemma}


\section{Proof of Theorem \ref{Thm1}}
\setcounter{equation}{0}
\renewcommand\theequation{4.\arabic{equation}}
\renewcommand\thetheorem{4.\arabic{theorem}}
In this section, we shall prove Theorem \ref{Thm1}. Firstly, we show the minimizer of $J(u)$ constrained on $\mathcal{S}_{c}\cap\mathcal{M}$ is attained.

\begin{lemma}\label{lem6}  Assume that $(f_{1})-(f_{5})$ hold. Then there exists $u_0\in \mathcal{S}_{c}\cap\mathcal{M}$ such that $J(u_0)=m(c).$
\end{lemma}
\begin{proof} By Lemma \ref{lem4} and the Ekeland variational principle,
there exists a minimizing sequence $\{u_{n}\}\subset \mathcal{S}_{c}\cap\mathcal{M}$ such that
\begin{equation}\label{8-27-2}
J(u_{n})\to m(c)~~\mbox{as}~~n\to+\infty.
\end{equation}
By similar arguments as in the proof of Lemma \ref{conti}, we can get that $\{u_{n}\}$ is bounded in $H^{1}(\mathbb{R}^{N}).$
We claim that $\{u_{n}$\} is non-vanishing.
We assume by contradiction that $\{u_{n}$\} is vanishing. Then by the Lions lemma \cite{Lions84}, we deduce
\begin{equation}\label{4.1}
\int_{\mathbb{R}^{N}}|u_{n}|^{p_{1}}dx\rightarrow 0.
\end{equation}
Then, by $(f_{1})-(f_{2}), (f_{4})-(f_{5}),$ $\mathcal{P}(u_{n})=0$ and $(\ref{2.15}),$ for any $\epsilon>0,$ there exist $p<p_{1}<2^{*},$ $C_{\epsilon, p_{1}}>0$ such that
\begin{eqnarray}
\int_{\mathbb{R}^{N}}|\nabla u_{n}|^{2}dx
& \leq& \frac{N}{2}(2^{*}-2)\int_{\mathbb{R}^{N}}F(u_{n})dx \nonumber\\
& \leq& 2^{*}C_{\epsilon,p_1}\int_{\mathbb{R}^{N}}|u_{n}|^{p_{1}}dx + 2^{*}(\eta+\epsilon)\int_{\mathbb{R}^{N}}|u_{n}|^{2^{*}}dx\nonumber\\
&\leq &2^{*}(\eta+\epsilon)S^{\frac{N}{2-N}}\Big(\int_{\mathbb{R}^{N}}|\nabla u_{n}|^{2}dx\Big)^{\frac{N}{N-2}}+o_n(1). \label{8-27-1}
\end{eqnarray}
Since $\{u_n\}$ is bounded in
$H^{1}(\mathbb{R}^{N})$, by Lemma \ref{lem2}, we may assume that, up to a subsequence, $\|\nabla u_n\|_2^2\to l^*>0$.
By (\ref{8-27-1}), we obtain $l^*\geq\left(2^{*}(\eta+\epsilon)\right)^{\frac{2-N}{2}}S^{\frac{N}{2}}.$ Moreover, by $(f_{1})-(f_{2}), (f_{4}),$ (\ref{8-27-2})-(\ref{4.1}) and Lemma \ref{lem1}, for any $\epsilon>0,$ there exist $p<p_{1}<2^{*}, C_{\epsilon,p_{1}}>0$ such that for any $t>0,$ we have
\begin{align*}
& m(c)+o_n(1) \nonumber\\
& \geq  \frac{t^{2}}{2}\int_{\mathbb{R}^{N}}|\nabla u_{n}|^{2}dx-t^{\frac{N}{2}(p_{1}-2)}C_{\epsilon,p_{1}}\int_{\mathbb{R}^{N}}|u_{n}|^{p_{1}}dx-t^{2^{*}}(\eta+\epsilon)\int_{\mathbb{R}^{N}}|u_{n}|^{2^{*}}dx \nonumber\\
& \geq  \frac{t^{2}}{2}\int_{\mathbb{R}^{N}}|\nabla u_{n}|^{2}dx
-t^{2^{*}}(\eta+\epsilon) S^{\frac{N}{2-N}}\Big(\int_{\mathbb{R}^{N}}|\nabla u_{n}|^{2}dx\Big)^{\frac{2^{*}}{2}}.
\end{align*}
Hence, taking $t=\left(2^{*}(\eta+\epsilon)\right)^{\frac{2-N}{4}}S^{\frac{N}{4}}(l^*)^{-\frac{1}{2}}$ and $n\to+\infty$, by (\ref{4.1}), we obtain $m(c)\geq \frac{1}{N}\left(2^{*}(\eta+\epsilon)\right)^{\frac{2-N}{2}}S^{\frac{N}{2}},$ which is contrary to (\ref{2.21}).
Therefore the claim holds. Thus there exists $u_0\in H^{1}(\mathbb{R}^{N})\setminus\{0\}$ such that up to a subsequence $u_{n}\rightharpoonup u_0$ in $H^{1}(\mathbb{R}^{N}).$
Denote $u_{n,0}=u_{n}-u_{0}.$
By the Brezis-Lieb lemma, we have
\begin{equation*}\label{27-3}
  \|u_{n}\|^{2}_{2}= \|u_{0}\|^{2}_{2}+ \|u_{n,0}\|^{2}_{2}+o_{n}(1)
 \end{equation*}
 and
 \begin{equation}\label{4-1}
 \|\nabla u_{n}\|^{2}_{2}= \|\nabla u_{0}\|^{2}_{2}+ \|\nabla u_{n,0}\|^{2}_{2}+o_{n}(1).
 \end{equation}
For convenience, we let $F_1(s)=F(s)-\eta|s|^{2^{*}}$. By $(f_1)-(f_2)$ we deduce
\begin{equation*}
  \lim\limits_{s\to 0}\frac{F_1(s)}{|s|^{2+\frac{4}{N}}}=0 \quad \mbox{and} \quad \lim\limits_{s\to +\infty}\frac{F_1(s)}{|s|^{2^{*}}}=0.
\end{equation*}
Denote $f_1=F_1'.$ Clearly $f_1\in C^{1}(\mathbb{R})$ and
\begin{equation*}
  \lim\limits_{s\to 0}\frac{f_1(s)}{|s|^{1+\frac{4}{N}}}=0 \quad \mbox{and} \quad \lim\limits_{s\to +\infty}\frac{f_1(s)}{|s|^{2^{*}-1}}=0.
\end{equation*}
By the Brezis-Lieb lemma again, using the boundedness of $\{u_{n}\}$ in $H^{1}(\mathbb{R}^{N})$, we get
\begin{equation*}
\int_{\mathbb{R}^{N}}\big(F_1(u_{n})-F_1(u_{n,0})-F_1(u_{0})\big)dx=o_{n}(1).
\end{equation*}
Then, in view of
\begin{equation}\label{4-3}
  \|u_{n}\|^{2^{*}}_{2^{*}}= \|u_{0}\|^{2^{*}}_{2^{*}}+ \|u_{n,0}\|^{2^{*}}_{2^{*}}+o_{n}(1),
 \end{equation}
it follows that
\begin{equation}\label{10-18-3}
\int_{\mathbb{R}^{N}}\big(F(u_{n})-F(u_{n,0})-F(u_{0})\big)dx=o_{n}(1),
\end{equation}
which together with (\ref{4-1}) gives
\begin{equation}\label{4-2}
J(u_{n})=J(u_{0})+J(u_{n,0})+o_{n}(1).
\end{equation}
By similar arguments as (\ref{10-18-3}), we obtain
\begin{equation}\label{10-18-1}
  \int_{\mathbb{R}^{N}}H(u_{n})dx=\int_{\mathbb{R}^{N}}H(u_{0})dx+\int_{\mathbb{R}^{N}}H(u_{n,0})dx + o_{n}(1).
\end{equation}
Indeed, defining $H_1(s)=f_1(s)s-2F_1(s), \forall s\in\mathbb{R}$, we have
\begin{equation*}
  \lim\limits_{s\to 0}\frac{H_1(s)}{|s|^{2+\frac{4}{N}}}=0 \quad \mbox{and} \quad \lim\limits_{s\to +\infty}\frac{H_1(s)}{|s|^{2^{*}}}=0.
\end{equation*}
Then
\begin{equation*}
  \lim\limits_{s\to 0}\frac{h_1(s)}{|s|^{1+\frac{4}{N}}}=0 \quad \mbox{and} \quad \lim\limits_{s\to +\infty}\frac{h_1(s)}{|s|^{2^{*}-1}}=0,
\end{equation*}
where $h_1=H_1'.$ Thus
\begin{equation*}
\int_{\mathbb{R}^{N}}\big(H_1(u_{n})-H_1(u_{n,0})-H_1(u_{0})\big)dx=o_{n}(1).
\end{equation*}
We note that $H(s)=H_1(s)+\eta(2^{*}-2)|s|^{2^{*}},$ using (\ref{4-3}) yields (\ref{10-18-1}).
Therefore we obtain
\begin{equation}\label{10-17-1}
0=\mathcal{P}(u_{n})=\mathcal{P}(u_{0})+\mathcal{P}(u_{n,0})+o_{n}(1).
\end{equation}

In what follows, we claim that $\mathcal{P}(u_{0})\leq 0.$ In fact, up to a subsequence if necessary, assuming $\alpha_n:=\int_{\mathbb{R}^{N}}|\nabla u_{n,0}|^{2}dx\rightarrow\alpha_{0}\geq0$, we divide into the following two cases:

\vspace*{4pt}\noindent\textbf{Case 1.} $\alpha_{0}=0.$ By (\ref{2.15}), we deduce
\begin{equation}\label{4.2}
\int_{\mathbb{R}^{N}} |u_{n,0}|^{q}dx\rightarrow 0 \quad \mbox{for} \ q\in(2,2^{*}).
\end{equation}
Furthermore, by $(f_{1})-(f_{2}), (f_{4}),$ for any $\epsilon>0,$ there exist $p<p_{1}<2^{*}$ and $C_{\epsilon,p_{1}}>0$ such that
\begin{equation}\label{estimate}
\int_{\mathbb{R}^{N}}F(u_{n,0})dx
  \leq \int_{\mathbb{R}^{N}}\left(C_{\epsilon,p_{1}}|u_{n,0}|^{p_{1}}+(\eta+\epsilon)|u_{n,0}|^{2^{*}}\right)dx.
\end{equation}
By (\ref{4.2})-(\ref{estimate}) and $(f_{5}),$ we get
\begin{equation*}
  0\leq \int_{\mathbb{R}^{N}}H(u_{n,0})dx
  \leq (2^{*}-2)\int_{\mathbb{R}^{N}}F(u_{n,0})dx
  \rightarrow 0,
\end{equation*}
which implies that $ \int_{\mathbb{R}^{N}}H(u_{n,0})dx\rightarrow 0$ as $n\rightarrow+\infty.$
Together with $\int_{\mathbb{R}^{N}}|\nabla u_{n,0}|^{2}dx\rightarrow 0,$ we get $\mathcal{P}(u_{n,0})\rightarrow 0$ as $n\rightarrow+\infty.$ Hence, by (\ref{10-17-1}), we obtain $\mathcal{P}(u_{0})=0$ and the claim holds.

\vspace*{4pt}\noindent\textbf{Case 2.} $\alpha_{0}>0.$ By contradiction, we assume that $\mathcal{P}(u_{0})> 0.$
By $(f_{1})-(f_{3}),$ we have
\begin{equation}\label{29-1}
  f(s)s\geq\big(2+\frac{4}{N}\big)F(s), \quad  s\in\mathbb{R}.
\end{equation}
Then
\begin{equation}\label{29-2}
  J(u_{0})=\frac{1}{2}\mathcal{P}(u_{0})+\frac{N}{4}\int_{\mathbb{R}^{N}}\left(f(u_{0})u_{0}-\big(2+\frac{4}{N}\big)F(u_{0})\right)dx>0.
\end{equation}
In view of (\ref{10-17-1}), we get $\mathcal{P}(u_{n,0})\leq 0.$ Then by Lemma \ref{lem1} and $(f_{3}),$ there exists $t_{u_{n,0}}\in (0,1]$ such that $\mathcal{P}(t_{u_{n,0}}\star u_{n,0})= 0.$
Furthermore, we have
\begin{align*}
 & J(u_{n,0})-J(t_{u_{n,0}}\star u_{n,0}) \nonumber\\
&= \frac{1-t_{u_{n,0}}^{2}}{2}\int_{\mathbb{R}^{N}}|\nabla u_{n,0}|^{2}dx+t_{u_{n,0}}^{-N}\int_{\mathbb{R}^{N}}F(t_{u_{n,0}}^{\frac{N}{2}}u_{n,0})dx-\int_{\mathbb{R}^{N}}F(u_{n,0})dx \nonumber\\
&= \int_{\mathbb{R}^{N}}\Bigg[\frac{N\big(1-t_{u_{n,0}}^{2}\big)}{4}f(u_{n,0})u_{n,0}-\Big(1+\frac{N\big(1-t_{u_{n,0}}^{2}\big)}{2}\Big)F(u_{n,0})
 \nonumber\\
 &\quad+t_{u_{n,0}}^{-N}F(t_{u_{n,0}}^{\frac{N}{2}}u_{n,0})\Bigg]dx +\frac{1-t_{u_{n,0}}^{2}}{2}\mathcal{P}(u_{n,0})\nonumber\\
&=\int_{\mathbb{R}^{N}}\int_{t_{u_{n,0}}}^{1}\left[\frac{N}{2}t|u_{n,0}|^{2+\frac{4}{N}}\Bigg(\frac{H(u_{n,0})}{|u_{n,0}|^{2+\frac{4}{N}}}-\frac{H(t^{\frac{N}{2}}u_{n,0})}{|t^{\frac{N}{2}}u_{n,0}|^{2+\frac{4}{N}}}\Bigg)\right]dtdx \nonumber\\
&\quad +\frac{1-t_{u_{n,0}}^{2}}{2}\mathcal{P}(u_{n,0})\geq  \frac{1-t_{u_{n,0}}^{2}}{2}\mathcal{P}(u_{n,0}).
\end{align*}
Denote $c_{n,0}:=\|u_{n,0}\|^{2}_{2}.$ Clearly, $c_{n,0}\leq c.$ Then by (\ref{10-18-3}), (\ref{10-18-1}), (\ref{29-1}) and Lemma \ref{nonincrea}, we obtain
\begin{align*}
m(c)&=\lim\limits_{n\rightarrow +\infty}\left(J(u_{n})-\frac{1}{2}\mathcal{P}(u_{n})\right)\\
&=\lim\limits_{n\rightarrow +\infty}\left(\frac{N}{4}\int_{\mathbb{R}^{N}}H(u_{n})dx-\int_{\mathbb{R}^{N}}F(u_{n})dx\right)\\
&=\frac{N}{4}\int_{\mathbb{R}^{N}}\left(f(u_{0})u_{0}-\big(2+\frac{4}{N}\big)F(u_{0})\right)dx+\lim\limits_{n\to+\infty}\left(J(u_{n,0})-\frac{1}{2}\mathcal{P}(u_{n,0})\right)\\
&\geq \lim\limits_{n\to+\infty}\left(J(u_{n,0})-\frac{1}{2}\mathcal{P}(u_{n,0})\right)\\
&\geq \lim\limits_{n\to+\infty}\left(J(t_{u_{n,0}}\star u_{0})-\frac{t_{u_{n,0}}^{2}}{2}\mathcal{P}(u_{n,0})\right)\\
&\geq \lim\limits_{n\to+\infty}J(t_{u_{n,0}}\star u_{0})\geq \lim\limits_{n\to+\infty}m(c_{n,0})\geq m(c),
\end{align*}
which implies that $\lim\limits_{n\to+\infty}\mathcal{P}(u_{n,0})=0$ and
\begin{equation}\label{27-1}
\lim\limits_{n\to+\infty}J(u_{n,0})=\lim\limits_{n\to+\infty}m(c_{n,0})=m(c).
\end{equation}
On the other hand, by (\ref{8-27-2}) and (\ref{4-2}), we have
\begin{equation*}
m(c)=J(u_{n})+o_{n}(1)= J(u_{0}) + J(u_{n,0})+o_{n}(1).
\end{equation*}
Therefore, by (\ref{29-2}) and (\ref{27-1}),
it follows that
\begin{equation*}
m(c)>m(c)-J(u_{0})=\lim\limits_{n\to+\infty}J(u_{n,0})=\lim\limits_{n\to+\infty}m(c_{n,0})=m(c),
\end{equation*}
which produces a contradiction and then the claim holds.

Since $\mathcal{P}(u_{0})\leq 0,$ as before, one can see that there exists $t_0\in (0,1]$ such that $t_0\star u_0\in \mathcal{S}_{c_0}\cap \mathcal{M}$ and
\begin{equation}\label{4.5}
J(u_0)-J(t_0\star u_0)\geq \frac{1-t_0^{2}}{2}\mathcal{P}(u_0).
\end{equation}
Denote $c_0=\int_{\mathbb{R}^{N}}|u_0|^{2}dx.$ Clearly, $c_0\in (0, c].$
Therefore, by (\ref{10-18-3}), (\ref{10-18-1}), (\ref{29-1}), (\ref{4.5}) and  Lemma \ref{nonincrea}, we obtain
\begin{eqnarray*}\label{4.7}
&&\hspace*{-0.6cm}m(c) \nonumber\\
&=&\lim_{n\rightarrow \infty}\left(J(u_{n})-\frac{1}{2}\mathcal{P}(u_{n})\right)\nonumber\\
&=&\lim\limits_{n\rightarrow +\infty}\left(\frac{N}{4}\int_{\mathbb{R}^{N}}H(u_{n})dx-\int_{\mathbb{R}^{N}}F(u_{n})dx\right)\nonumber\\
&=&J(u_0)-\frac{1}{2}\mathcal{P}(u_0)+\lim\limits_{n\to+\infty}\frac{N}{4}\int_{\mathbb{R}^{N}}\Big(f(u_{n,0})u_{n,0}-\big(2+\frac{4}{N}\big)F(u_{n,0})\Big)dx\nonumber\\
&\geq& J(t_0\star u_0)-\frac{t_0^{2}}{2}\mathcal{P}(u_0) \geq m(c_0)\ge m(c)
\end{eqnarray*}
which implies $m(c_{0})=m(c)$ and $\mathcal{P}(u_{0})=0,$ that is $t_{0}=1.$ Then we obtain $u_0\in \mathcal{S}_{c}\cap \mathcal{M}$ and $J(u_0)=m(c_0).$
Using Lemma \ref{decrea} at $c_{0}$ and $m(c_{0})=m(c),$ we deduce $c_{0}=c$ and thus $J(u_0)=m(c).$
\end{proof}

\begin{proof}[{Completion of the proof of Theorem \ref{Thm1}}.]
Set $E(u):=J(t_{u}\star u),$ where $u\in\mathcal{S}_{c}$ and $t_{u}\star u\in\mathcal{M}.$ By Lemma \ref{lem1} and Lemma \ref{lem4}, we get
$$m(c)=\inf\limits_{u\in\mathcal{S}_{c}\cap\mathcal{M}}J(u)=\inf\limits_{u\in\mathcal{S}_{c}}\max\limits_{t>0}J(u)
      =\inf\limits_{u\in\mathcal{S}_{c}}J(t_{u}\star u)=\inf\limits_{u\in\mathcal{S}_{c}}E(u).$$
Using Lemma \ref{lem6}, we obtain $u_{0}\in\mathcal{S}_{c}\cap \mathcal{M}$ such that $J(u_{0})=m(c).$ Then there exists $v_{0}\in\mathcal{S}_{c}$ such that $t_{v_{0}}\star v_{0}=u_{0}$ and $E(v_{0})=J(t_{v_{0}}\star v_{0})=J(u_{0})=m(c),$ which implies $v_{0}$ is a minimizer of $E|_{\mathcal{S}_{c}}.$

We claim that for any $u\in\mathcal{S}_{c}$ and $\psi\in T_{u}\mathcal{S}_{c},$ $E\in \mathcal{C}^{1}(\mathcal{S}_{c},\mathbb{R})$ and
\begin{equation}\label{9-1}
J'(t_{u}\star u)(t_{u}\star\psi)=E'(u)\psi.
\end{equation}
For $|s|$ small enough, we estimate $E(u+s\psi)-E(u)$.
By Lemma \ref{lem1}, for any $u\in\mathcal{S}_{c},$ there exists $t_{u}$ such that $t_{u}\star u\in\mathcal{M}$ and $J(t_{u}\star u)>J(t\star u)$ for any $t\neq t_{u}.$
Then, for any $\psi\in T_{u}\mathcal{S}_{c},$ by the mean value theorem we obtain
\begin{align*}
& E(u+s\psi)-E(u)\nonumber\\
&=J\big(t_{u+s\psi}\star(u+s\psi)\big)-J(t_{u}\star u) \nonumber\\
&\leq J\big(t_{u+s\psi}\star(u+s\psi)\big)-J\left(t_{u+s\psi}\star u\right) \nonumber\\
&= \frac{1}{2}t_{u+s\psi}^{2}\int_{\mathbb{R}^{N}}\Big( 2s\nabla u\cdot\nabla \psi+s^{2}|\nabla \psi|^{2}\Big)dx\nonumber\\
&\quad-t_{u+s\psi}^{-N}\int_{\mathbb{R}^{N}}\Big[f\left(t_{u+s\psi}^{\frac{N}{2}}\left(u+\xi_{s}s\psi\right)\right)t_{u+s\psi}^{\frac{N}{2}}s\psi \Big] dx,
\end{align*}
where $s\in[0,1]$ and $\xi_{s}\in(0,1).$ Similarly,
\begin{align*}
&E(u+s\psi)-E(u) \nonumber\\
&=J\left(t_{u+s\psi}\star(u+s\psi)\right)-J(t_{u}\star u) \nonumber\\
&\geq J\big(t_{u}\star(u+s\psi)\big)-J(t_{u}\star u) \nonumber\\
&=\frac{1}{2}t_{u}^{2}\int_{\mathbb{R}^{N}}\Big( 2s\nabla u\cdot\nabla \psi+s^{2}|\nabla \psi|^{2}\Big)dx-t_{u}^{-N}\int_{\mathbb{R}^{N}}f\left(t_{u}^{\frac{N}{2}}(u+\zeta_{s}s\psi)\right)t_{u}^{\frac{N}{2}}s\psi dx.
\end{align*}
By Lemma \ref{lem3} (i), we get $\lim\limits_{s\rightarrow 0}t_{u+s\psi}=t_{u}.$
Hence, we have
$$\lim\limits_{s\rightarrow 0}\frac{E(u+s\psi)-E(u)}{s}
=t_{u}^{2}\int_{\mathbb{R}^{N}}\nabla u \cdot \nabla \psi dx-t_{u}^{-N}\int_{\mathbb{R}^{N}}f(t_{u}^{\frac{N}{2}}u)t_{u}^{\frac{N}{2}}\psi dx.$$
By $(f_{1}),$ the H\"{o}lder inequality and Gagliardo-Nirenberg inequality, it follows that the G\^{a}teaux derivative of $E$ is bounded linear in $\psi.$ In view of Lemma \ref{lem3} (i), it is continuous in $u.$ Therefore, by Proposition 1.3 in \cite{W1996}, we deduce $E: \mathcal{S}_{c}\rightarrow \mathbb{R}$ is of class $\mathcal{C}^{1}.$ Furthermore, by direct computations,
\begin{eqnarray*}
E'(u)(\psi)&=&\int_{\mathbb{R}^{N}}\nabla(t_{u}\star u) \cdot \nabla(t_{u}\star\psi) dx-\int_{\mathbb{R}^{N}}f(t_{u}\star u)(t_{u}\star\psi)dx\nonumber\\
&=&J'(t_{u}\star u)(t_{u}\star \psi).
\end{eqnarray*}
Hence the claim holds.

Now, by (\ref{9-1}) we deduce
\begin{displaymath}
\begin{aligned}
\|dJ(u_{0})\|_{(T_{u_{0}}\mathcal{S}_{c})^{*}}
&=\sup\limits_{\phi\in T_{u_{0}}\mathcal{S}_c, \|\phi\|\le1} \Big| dJ(u_{0})[\phi]\Big|\\
&=\sup\limits_{\phi\in T_{u_{0}}\mathcal{S}_c, \|\phi\|\le1} \Big|dJ(t_{v_{0}}\star v_{0})\big[t_{v_{0}}\star(t^{-1}_{v_{0}}\star\phi)\big]\Big|\\
&=\sup\limits_{\phi\in T_{u_{0}}\mathcal{S}_c, \|\phi\|\le1} \Big|dE(v_{0})\big[t_{v_{0}}^{-1}\star\phi\big]\Big|\\
&\leq\|dE(v_0)\|_{(T_{v_{0}}\mathcal{S}_{c})^{*}}\cdot\sup\limits_{\phi\in T_{u_0}\mathcal{S}_c, \|\phi\|\le1} \|t_{v_0}^{-1}\star\phi\| \\
&=\|dE(v_0)\|_{(T_{v_{0}}\mathcal{S}_{c})^{*}}\cdot\sup\limits_{\phi\in T_{u_0}\mathcal{S}_c, \|\phi\|\le1}\Big(t_{v_0}^{-1}\|\nabla \phi\|_{2}+\|\phi\|_{2}\Big)\\
&\leq \max\left\{t_{v_{0}}^{-1},1\right\}\|dE(v_{0})\|_{(T_{v_{0}}\mathcal{S}_{c})^{*}}=0.
\end{aligned}
\end{displaymath}
It follows that
$u_{0}$ is a critical point of $J|_{\mathcal{S}_c}.$ Using $(f_{5})$, by standard arguments it follows that, for some $\lambda<0$,
 $u_0$ weakly solves (\ref{eq01}). In view of $J(u_0)=\inf\limits_{u\in\mathcal{S}_c}J(u)=m(c)$, we infer
that $u_{0}$ is a normalized ground state solution of problem (\ref{eq01}).
This completes the proof.
\end{proof}

\medskip
Received February 2024; revised March 2024; early access March 2024.
\medskip


\begin{thebibliography}{99}

\bibitem{Alves2022} (MR4350192) [10.1007/s00526-021-02123-1]
\newblock C. O. Alves, C. Ji and O. H. Miyagaki,
\newblock \doititle{Normalized solutions for a Schr\"{o}dinger equation with critical growth in $\mathbb{R}^{N}$},
\newblock \emph{Calc. Var. Partial Differential Equations}, \textbf{61} (2022), Paper No. 18, 24 pp.





\bibitem{Bart2013} (MR3009665) [10.1007/s00013-012-0468-x]
\newblock T. Bartsch and S. de Valeriola,
\newblock \doititle{ Normalized solutions of nonlinear Schr\"{o}dinger equations},
\newblock \emph{Arch. Math.}, \textbf{100} (2013), 75-83.

\bibitem{Bart2017} (MR3639521) [10.1016/j.jfa.2017.01.025]
\newblock T. Bartsch and N. Soave,
\newblock \doititle{A natural constraint approach to normalized solutions on nonlinear Schr\"{o}dinger equations and systems},
\newblock \emph{J. Funct. Anal.}, \textbf{272} (2017), 4998-5037.


\bibitem{BS2017-2} (MR3802492) [10.1016/j.jfa.2018.02.007]
\newblock T. Bartsch and N. Soave,
\newblock \doititle{Correction to a natural constraint approach to normalized solutions on nonlinear Schr\"{o}dinger equations and systems [J. Funct. Anal. 272 (2017) 4998-5037]},
\newblock \emph{J. Funct. Anal.}, \textbf{275} (2018), 516-521.


\bibitem{Bieg2021} (MR4232669) [10.1016/j.jfa.2021.108989]
\newblock B. Bieganowski and J. Mederski,
\newblock \doititle{Normalized ground states of the nonlinear Schr\"{o}dinger equation with at least mass critical growth},
\newblock \emph{J. Funct. Anal.}, \textbf{280} (2021), Paper No. 108989, 26 pp.

\bibitem{Caze82} (MR677997) [10.1007/BF01403504]
\newblock T. Cazenave and P. L. Lions,
\newblock \doititle{Orbital stability of standing waves for some nonlinear Schr\"odinger equations},
\newblock \emph{Comm. Math. Phys.}, \textbf{85} (1982), 549-561.

\bibitem{Chang23} (MR4531060) [10.1007/s12220-022-01130-8]
\newblock X. Chang, M. Liu and D. Yan,
\newblock \doititle{Normalized ground state solutions of nonlinear Schr\"odinger equations involving exponential critical growth},
\newblock \emph{J. Geom. Anal.}, \textbf{33} (2023), Paper No. 83, 20 pp.

\bibitem{Chen23} (MR4662421) [10.1007/s00526-023-02592-6]
\newblock S. Chen and X. Tang,
\newblock \doititle{Normalized solutions for Schr\"odinger equations with mixed dispersion and critical exponential growth in $\mathbb{R}^{2}$},
\newblock \emph{Calc. Var. Partial Differential Equations}, \textbf{62} (2023), Paper No. 261, 37 pp.

\bibitem{Ikom2019} (MR4021261)
\newblock N. Ikoma and K. Tanaka,
\newblock A note on deformation argument for $L^{2}$ normalized solutions of nonlinear Schr\"{o}dinger equations and systems,
\newblock \emph{Adv. Differential Equations}, \textbf{24} (2019), 609-646.

\bibitem{Jean97} (MR1430506) [10.1016/S0362-546X(96)00021-1]
\newblock L. Jeanjean,
\newblock \doititle{Existence of solutions with prescribed norm for semilinear elliptic equations},
\newblock \emph{Nonlinear Anal.}, \textbf{28} (1997), 1633-1659.

\bibitem{Jean(2)2020}   [10.1016/j.matpur.2022.06.005]
\newblock L. Jeanjean, J. Jendrej, T. T. Le and N. Visciglia,
\newblock \doititle{Orbital stability of ground states for a Sobolev critical Schr\"odinger equation},
\newblock \emph{J. Math. Pures Appl.}, \textbf{164} (2022), 158-179.

\bibitem{Jean2021} (MR4476243) [10.1007/s00208-021-02228-0]
\newblock L. Jeanjean and T. T. Le,
\newblock \doititle{Multiple normalized solutions for a Sobolev critical Schr\"odinger equations},
\newblock \emph{Math. Ann.}, \textbf{384} (2022), 101-134.

\bibitem{Jean(1)2020}  [10.1007/s00526-020-01828-z]
\newblock L. Jeanjean and S.-S. Lu,
\newblock \doititle{A mass supercritical problem revisited},
\newblock \emph{Calc. Var. Partial Differential Equations}, \textbf{59} (2020), Paper No. 174, 43 pp.

\bibitem{Li2021} (MR4290382) [10.1007/s00526-021-02020-7]
\newblock X. Li,
\newblock \doititle{Existence of normalized ground states for the Sobolev critical Schr\"{o}dinger equation with combined nonlinearities},
\newblock \emph{Calc. Var. Partial Differential Equations}, \textbf{60} (2021), Paper No. 169, 14 pp.

\bibitem{Lions84} (MR778974) [10.1016/s0294-1449(16)30422-x]
\newblock P.-L. Lions,
\newblock \doititle{The concentration-compactness principle in the calculus of variations. The locally compact case. \uppercase\expandafter{\romannumeral2}},
\newblock \emph{Ann. Inst. H. Poincar\'{e} Anal. Non Lin\'{e}aire}, \textbf{1} (1984), 223-283.

\bibitem{Mede2020} (MR4173560) [10.1088/1361-6544/aba889]
\newblock J. Mederski,
\newblock \doititle{Nonradial solutions of nonlinear scalar field equations},
\newblock \emph{Nonlinearity}, \textbf{33} (2020), 6349–6380.

\bibitem{Mede2022} (MR4344574) [10.1007/s00526-021-02116-0]
\newblock J. Mederski and J. Schino,
\newblock \doititle{Least energy solutions to a cooperative system of Schr\"odinger equations with prescribed $L^2$-bounds: at least $L^2$-critical growth},
\newblock \emph{Calc. Var. Partial Differential Equations}, \textbf{61} (2022), Paper No. 10, 31 pp.

\bibitem{Shib2014} (MR3147450) [10.1007/s00229-013-0627-9]
\newblock M. Shibata,
\newblock \doititle{Stable standing waves of nonlinear Schr\"{o}dinger equations with a general nonlinear term},
\newblock \emph{Manuscripta Math.}, \textbf{143} (2014), 221-237.

\bibitem{Soave(2)2020}  [10.1016/j.jfa.2020.108610]
\newblock N. Soave,
\newblock \doititle{Normalized ground states for the NLS equation with combined nonlinearties: the Sobolev critical case},
\newblock \emph{J. Funct. Anal.}, \textbf{279} (2020), Paper No. 108610, 43 pp.

\bibitem{Wei2021} (MR4433054) [10.1016/j.jfa.2022.109574]
\newblock J. Wei and Y. Wu,
\newblock \doititle{Normalized solutions for Schr\"{o}dinger equations with critical sobolev exponent and mixed nonlinearities},
\newblock \emph{J. Funct. Anal.}, \textbf{283} (2022), Paper No. 109574, 46 pp.

\bibitem{W1983} (MR691044)
\newblock M. Weinstein,
\newblock  Nonlinear Schr\"odinger equations and sharp interpolation estimates,
\newblock \emph{Comm. Math. Phys.}, \textbf{87} (1982/83), 567-576.

\bibitem{W1996} (MR1400007) [10.1007/978-1-4612-4146-1]
\newblock M. Willem,
\newblock \doititle{\emph{Minimax Theorems}},
\newblock Birkh\"auser Verlag, Boston, 1996.

\bibitem{Yang2020} (MR4134927) [10.1007/s00013-020-01468-x]
\newblock Z. Yang,
\newblock \doititle{A new observation for the normalized solution of the Schr\"{o}dinger equation},
\newblock \emph{Arch. Math.}, \textbf{115} (2020), 329-338.

\end{thebibliography}
\end{document}